\documentclass[a4paper,oneside, reqno]{amsart}
\usepackage[utf8]{inputenc}
\usepackage{amsmath, amssymb, amsthm, enumerate}
\usepackage{hyperref}	%allows urls
\usepackage[pdftex]{graphicx}
\usepackage{asymptote}
\usepackage{indentfirst}
\usepackage{mathrsfs}

\usepackage{tikz-cd}
\usepackage{todonotes}
\usepackage{verbatim}

\usepackage{lineno,hyperref}
\modulolinenumbers[5]

\makeatletter
\newtheorem*{rep@theorem}{\rep@title}
\newcommand{\newreptheorem}[2]{%
\newenvironment{rep#1}[1]{%
 \def\rep@title{#2~\ref{##1}}%
 \begin{rep@theorem}}%
 {\end{rep@theorem}}}
\makeatother

\theoremstyle{plain}
\newtheorem{theorem}{Theorem}[section]
\newtheorem{lemma}[theorem]{Lemma}

\newtheorem{claim}[theorem]{Claim}

\newreptheorem{thm}{Theorem}

\newtheorem{lem}[theorem]{Lemma}
\newtheorem{prop}[theorem]{Proposition}
\newreptheorem{prop}{Proposition}
\newtheorem{cor}[theorem]{Corollary}
\newreptheorem{cor}{Corollary}

\newreptheorem{conj}{Conjecture}

\theoremstyle{definition}

\theoremstyle{definition}

\theoremstyle{remark}

\newtheorem*{rem}{Remark}

\numberwithin{equation}{section}

\newcommand{\FF}{\mathbb{F}}

\newcommand{\NN}{\mathbb{N}}

\newcommand{\RR}{\mathbb{R}}

\newcommand{\ZZ}{\mathbb{Z}}

\DeclareMathOperator{\rank}{rank}
\DeclareMathOperator{\chr}{char}

\DeclareMathOperator{\Supp}{Supp}

\title[shift operator-based polynomial method in additive combinatorics]{A new shift operator-based polynomial method in additive combinatorics}
\author{Sammy Luo}
\thanks{Department of Mathematics, MIT, Cambridge, MA 02139, USA. \\
\indent This material is based upon work supported by NSF Award No. 2303290, as well as work supported by NSF GRFP Grant DGE-1656518. Email: {\tt sammyluo@mit.edu}.}

\begin{document}

\maketitle

\begin{abstract}
    We introduce a new form of the polynomial method based on what we call ``shift operators,'' which we use to give efficient and intuitive new proofs of results previously shown using a wide range of polynomial methods, including Alon's Combinatorial Nullstellensatz and the Croot-Lev-Pach method. We end by discussing some potential new directions in which the tools introduced here may be fruitfully applied.
\end{abstract}

\section{Introduction}

Alon's Combinatorial Nullstellensatz \cite{alonCNS} has long served as a useful tool in additive combinatorics, offering simple and elegant proofs for classical results like the Erd\H{o}s-Ginzburg-Ziv theorem \cite{EGZ} and the Chevalley-Warning theorem \cite{CW}, as well as a wide range of new results (see e.g. \cite{Sun2008,Seamone2014,ABNR2022,JN2022,CH2023}). It is most often used in its ``Non-Vanishing Lemma'' form, stated below.

\begin{theorem}[Combinatorial Nullstellensatz,~{\cite{alonCNS}}]
\label{thm:null1}
Let $\FF$ be an arbitrary field, and let $f=f(x_1,\dots,x_n)$ be a polynomial in $\FF[x_1,\dots,x_n]$. Suppose the degree $deg(f)$ of $f$ is $\sum_{i=1}^n t_i$,  where each $t_i$ is a nonnegative integer, and suppose the coefficient of $\prod_{i=1}^n x_i^{t_i}$ is nonzero. Then, if $S_1,\dots,S_n$ are subsets of $\FF$ with $|S_i|>t_i$, there are $s_1\in S_1,s_2\in S_2,\dots,s_n\in S_n$ so that $f(s_1,\dots,s_n)\neq 0$.
\end{theorem}

This non-vanishing lemma form can be generalized in several different ways. For example, the condition $\deg(f)=\sum_{i=1}^n t_i$ can be replaced with the condition that $\prod_{i=1}^n x_i^{t_i}$ is a maximal monomial in $f$ (i.e.~there is no other term $\prod_{i=1}^n x_i^{t_i'}$ in $f$ where $t_i'\geq t_i$ for all $i$) \cite{lasonCNSmaximal}; the sets $S_i$ can be replaced with multisets \cite{KosRonyaiCNS}, or the zeros at the points in $S_1\times\cdots\times S_n$ can be given multiplicity \cite{ballserraCNSpuncture}. A recent result by Xu, Han, and Kan \cite{XKH2023grobner} uses the framework of Gr\H{o}bner bases to provide a further common generalization of several of these variants. For a more detailed discussion of various versions of the Nullstellensatz, see Section~\ref{subsec:cns}.

The Nullstellensatz method, while ubiquitous, is far from the only version of the polynomial method that has proved useful in additive combinatorics. Other polynomial-based techniques useful in this area include Stepanov's method of auxiliary polynomials (see e.g. \cite{hp}) and Dvir's approach to the finite field Kakeya problem (especially the version of the argument known as the ``method of multiplicities'') \cite{dvir2009kakeya,dvir2009kakeyamult,guthkatz2010joints}, both of which involve making use of low degree polynomials that vanish with high multiplicity on a set of points.

In 2016, Croot, Lev, and Pach \cite{clp} introduced a new polynomial method, leading to a breakthrough by Ellenberg and Gijswijt \cite{ellenberggijswijt} on the ``capset problem,'' the case of the Erd\H{o}s-Ginzburg-Ziv problem over the field $\FF_3$. Their method has since been reformulated by Tao into the so-called ``slice rank method'' \cite{tao2016slicerankblog}, and has contributed to progress on many problems in additive combinatorics (see e.g. \cite{NaslundSawin2017Sunflower,lovaszsauermann2019multicolor,naslund2020egz,gijswijt2021excluding,costa2022bounds,sauermann2023erd}).

In this paper, we introduce a novel variant of the polynomial method based on what we call ``shift operators.'' These operators have a wide range of useful properties which simultaneously grant them impressive versatility and depth in their applications. Of particular note is the fact that our method is able to prove both a range of results traditionally proven using the Combinatorial Nullstellensatz, and several iconic applications of the Croot-Lev-Pach-Ellenberg-Gijswijt polynomial method. We remark that the group ring method of Petrov \cite{petrov2016groupring} similarly recovers the results of both of these methods; however, our method differs by focusing on the relationship of the operators we study with derivative operators. This leads into natural ways to handle multiplicity, allowing us to additionally encapsulate applications of even more variants of the polynomial method, such as Stepanov's method as used by Hanson and Petridis in \cite{hp}, and Dvir's method of multiplicities as used in the finite field Kakeya problem \cite{dvir2009kakeya,dvir2009kakeyamult}. By linking these seemingly disparate polynomial-based arguments, this new method holds promise for shedding light on previously unexplored connections among them, as well as the potential for new applications in additive combinatorics and beyond.

The structure of the paper is as follows. In Section~\ref{sec:prelim} we establish the notation we use, review definitions and basic facts about Hasse derivatives, and introduce the concept of shift operators. We then derive some useful properties of these shift operators and apply them to give new proofs of a wide range of results. We first specialize to the simpler one-dimensional case in Sections~\ref{sec:shift1d} and~\ref{sec:apply1d}, which is enough to prove several generalizations of the Nullstellensatz \cite{BatzayaBayarmagnai2020CNSgen, ballserraCNSpuncture} as well as the aforementioned result of Hanson and Petridis \cite{hp}. We then address the multi-dimensional case in Sections~\ref{sec:shiftmd} and~\ref{sec:applymd}, which allows us to make our connection with the Croot-Lev-Pach method, yielding proofs of results such as bounds on the sizes of multicolored sum-free sets (see \cite{lovaszsauermann2019multicolor}), as well as with Dvir's method. Finally, we discuss potential future directions of research in Section~\ref{sec:future}, including a promising approach towards further progress on the Erd\H{o}s-Ginzburg-Ziv problem over $\FF_p^n$.

\section{Preliminaries}

\label{sec:prelim}

\subsection{Motivating example: The Cauchy-Davenport Theorem}

\label{sec:cd}

Before introducing the technical definitions our method requires, we give a simple application to motivate them. The Cauchy-Davenport Theorem is the following classical result.
\begin{theorem}[Cauchy-Davenport, {\cite{CauchyDavenport}}]
If $p$ is a prime, and $A,B\subseteq \ZZ_p$ are nonempty, then
\[
|A+B|\geq \min(p,|A|+|B|-1).
\]
\end{theorem}
It is traditional \cite{alonCNS} for discussions of polynomial methods in additive combinatorics to give a proof of the Cauchy-Davenport theorem as one of their first example applications. Although there are many interesting and unrelated proofs of this result, some of which (including the original proof \cite{CauchyDavenport}, as well as Tao's uncertainty principle-based proof \cite{tao2005uncertainty}) do not use polynomial methods, the new proof we present here will aptly demonstrate some of the core ideas and intuition behind our methods.

\begin{proof}Alon's proof in \cite{alonCNS} starts by considering the polynomial
\[F(x,y)=\prod_{c\in C} (x+y-c)\]
over $\FF=\FF_p$, where $C\supset A+B$ with $|C|=|A|+|B|-2$. We begin similarly by defining the one-variable polynomial
\[
f(z)=\prod_{c\in A+B} (z-c),
\]
which by construction satisfies $f(a+b)=0$ for all $a\in A$ and $b\in B$. In particular, for any constant $b\in B$, the polynomial $f_b(x):=f(x+b)$ vanishes when $x=a$ for any $a\in A$. If we define $g_A(x)=\prod_{a\in A}(x-a)$, then $g_A(x)|f_b(x)$ for each $b\in B$.

Consider the vector space spanned by the set of polynomials $\{f_b(x)\}_{b\in B}=\{f(x+b)\}_{b\in B}$. This is a subspace of the space $V$ of polynomials $p(x)$ of degree at most $\deg(f)=|A+B|$ that are divisible by $g_A(x)$. By considering the set of possible quotients $p(x)/g_A(x)$, which is the set of all polynomials of degree at most $\deg(f)-\deg(g_A)=|A+B|-|A|$, we see that $V$ has dimension $(|A+B|-|A|)+1$.

On the other hand, we claim that, if $|A+B|=\deg(f)<p$, the set of polynomials $\{f(x+b)\}_{b\in B}$ is linearly independent. This would then yield $|B|\leq \dim(V)=|A+B|-|A|+1$ when $|A+B|<p$, which rearranges to the desired inequality. The linear independence we seek follows from the following observation.

\begin{claim}
\label{claim:monomiallindep}
For $0\leq d<p$ and any subset $H\subset \FF_p$ of size $d+1$, the set of polynomials $\{(x+h)^d\}_{h\in H}$ is linearly independent.
\end{claim}
\begin{proof}[Proof of Claim~\ref{claim:monomiallindep}]
It suffices to show that the $(d+1)\times (d+1)$ matrix $M=[m_{ih}]_{0\leq i\leq d,\: h\in H}$ of coefficients of the polynomials $(x+h)^d$ is nonsingular, where $m_{ih}$ is the coefficient of $x^{d-i}$ in $(x+h)^d$, which is $\binom{d}{i} h^{i}$. Since $i\leq d<p$, we can remove the factors of $\binom{d}{i}$ from each row without affecting whether the matrix is singular. This leaves us with the matrix $[h^{i}]_{0\leq i\leq d, h\in H}$, which is a Vandermonde matrix, and thus nonsingular as needed.
\end{proof}
Fix a linear combination $\ell:=\sum_{b\in B} c_b f(x+b)$. Applying Claim~\ref{claim:monomiallindep} to an arbitrary superset of $B$ of size $|A+B|+1$, we know that $\bar{\ell}:=\sum_{b\in B} c_b (x+b)^{|A+B|}\neq 0$. Let $k$ be the largest integer such that the coefficient of $x^k$ in $\bar{\ell}$ is nonzero, so that $\sum_{b\in B} c_b b^{|A+B|-k}\neq 0$ but $\sum_{b\in B} c_b b^{j}=0$ for all $j<|A+B|-k$. Then for all $d<|A+B|$, the coefficient of $x^k$ in $\sum_{b\in B} c_b (x+b)^{d}$ is zero, and so the coefficient of $x^k$ in $\ell$ equals the coefficient of $x^k$ in $\bar \ell$, which is nonzero. Hence, $\ell\neq 0$, and thus $\{f(x+b)\}_{b\in B}$ is linearly independent as needed.
\end{proof}

The key observation that ``shifts'' of a polynomial are linearly independent with each other is the starting point around which many of our tools will be built. The last part of the above argument, where we showed the independence explicitly, follows more immediately and in greater generality from the methods we will develop in Section~\ref{sec:shift1d}; see Corollary~\ref{cor:shifts}.

\begin{rem}
\label{rem:derCD}
Observe that in the case $|A|=|B|$, another way to finish the proof from the observation that the set $\{f(x+b)\}_{b\in B}$ is linearly independent is the following: let $S=\{f(x+b)\}_{b\in B}\cup \{f'(x+b)\}_{b\in B}$. If $S$ is linearly independent, then we have
\[
|A+B|=\deg(f)\geq |S|-1=2|A|-1=|A|+|B|-1,
\]
as desired. Otherwise, there exist constants $c_b,\tilde c_b$, not all zero, such that
\[
\sum_{b\in B}c_b f(x+b)=\sum_{b\in B}\tilde c_b f'(x+b).
\]
Since $g_A|\sum_{b\in B}c_b f(x+b)$, we must also have $g_A|\sum_{b\in B}\tilde c_b f'(x+b)=\frac{d}{dx}(\sum_{b\in B}\tilde c_b f(x+b))$. But we also have $g_A|\sum_{b\in B}\tilde c_b f(x+b)$, so in fact $g_A^2 | \sum_{b\in B}\tilde c_b f(x+b)$, implying that $|A+B|=\deg f\geq 2\deg g_A=2|A|$, and we are again done.

Since this alternate ending still uses the fact that $g_A$ divides the linearly independent polynomials $f(x+b)$, its purpose is only to offer a slightly different viewpoint on the proof, highlighting a potential connection with the linear independence of shifts of derivatives that we study in Section~\ref{sec:shift1d}.
\end{rem}

\subsection{Definitions and notation}

Let $\FF$ be a field. We will generally use $p$ to denote $\chr(\FF)$, the characteristic of $\FF$, if it is nonzero. For integers $a\leq b$ let $[a,b]$ denote the set of integers between $a$ and $b$ inclusive. For elements $v_1, \dots,v_m$ of a vector space $V$, denote by $\langle v_1,\dots,v_m\rangle$ the linear span of these elements. For a fixed positive integer $n$, we will be considering the polynomial ring $P_n:=\FF[X_1,\dots,X_n]$. We will occasionally consider its subspaces $P_{n}^d$ of polynomials with degree at most $d$, as well as its subspaces $P_n^{d_1,\dots,d_n}$ of polynomials with degree in $x_i$ at most $d_i$ for each $1\leq i\leq n$.

Let $\NN$ denote the set of nonnegative integers. Whenever we consider an $n$-tuple $\alpha\in \NN^n$, let its components be given by $\alpha=(\alpha_1,\dots,\alpha_n)$. Define the \emph{weight} of $\alpha$ by $|\alpha|:=\sum_{i=1}^n \alpha_i$. For $\alpha,\beta\in \NN^n$, we say $\alpha\leq \beta$ if $\alpha_i\leq \beta_i$ for all $i\in [1,n]$. Let $\alpha!=\prod_{i=1}^n \alpha_i!$, and $\binom{\alpha}{\beta}=\prod_{i=1}^n \binom{\alpha_i}{\beta_i}$.

For any $\alpha\in \NN^n$, let $X^\alpha=\prod_{i=1}^n X_i^{\alpha_i}$. For $f\in \FF[X_1,\dots,X_n]$, let $[X^\alpha]f$ denote the coefficient of $X^\alpha$ in $f$. Define the total degree $\deg(f)$ to be the maximal weight over all $\alpha$ such that $[X^\alpha]f\neq 0$, and for $1\leq i\leq n$, define the $i$-degree $\deg_i(f)$ to be the maximal value of $\alpha_i$ over all such $\alpha$.

Let $\partial_{i}$ denote the (formal) partial differential operator with respect to $X_i$, and for any $\alpha \in \NN^n$, define $\partial^\alpha=\prod_{i=1}^n \partial_i^{\alpha_i}$. We call this the $\alpha$th (ordinary) derivative. 
Recall that the $\alpha$th \emph{Hasse derivative} of $f$ is defined (in e.g.~\cite{dvir2009kakeyamult}) by
\[
H^{(\alpha)}f(X)=[Z^\alpha]f(X+Z),
\]
that is, the coefficient of $Z^\alpha$ in $f(X+Z)$ when treated as a polynomial in $Z$. In particular, note that $H^{(\alpha)}x^\beta=\binom{\beta}{\alpha}x^{\beta-\alpha}$ for $\alpha,\beta\in \NN^n$. Note also that $H^{(\alpha)}H^{(\beta)}f(X)=\binom{\alpha+\beta}{\alpha} H^{(\alpha+\beta)}f(X) = H^{(\beta)}H^{(\alpha)}f(X)$, i.e. Hasse derivatives commute with each other as operators. For convenience, we let $H^{(\alpha)}f=0$ when $\alpha\in \ZZ^n\setminus \NN^n$ and $f\in P_n$. When $\FF$ has characteristic zero, the Hasse derivative is equivalent to the ordinary derivative up to a constant factor, given by
\begin{equation*}
    \label{eqn:hasse}
    \tag{$1$}
    H^{(\alpha)}f= \frac{1}{\alpha!}\partial^\alpha f.
\end{equation*}
The same holds over $\FF$ with characteristic $p\neq 0$ as long as $\alpha_i<p$ for all $i\in [1,n]$
. When \eqref{eqn:hasse}~holds, it can be advantageous to work with the ordinary derivative, which has simpler multiplicative properties as a linear operator and is easier to build intuition around. In any other case, however, working with Hasse derivatives is preferred in order to obtain the most general results possible. We will use both as appropriate in the arguments that follow. 

\subsection{Shift Operators}

For $h\in \FF^n$, we define the linear operator $T^h$ on the space of polynomials $P_n=\FF[X_1,\dots,X_n]$ by
\[
T^h(f)(X)=f(X+h).
\]
Call these the \emph{shift operators}. 
The map $h\mapsto T^h$ can be thought of as a representation of the group $\FF^n$ on the vector space $P_n$, where for each $d\geq 0$, the subspace $P_{n}^d$ of polynomials of degree at most $d$ is an invariant subspace
. The notation $T^h$ is chosen because in certain ways, the group element $h$ behaves like an exponent in determining properties of the shift operators. For example, it is clear that $T^a T^b=T^{a+b}$ for all $a,b\in \FF^n$. Furthermore, we have the following observation.

\begin{claim}
\label{claim:exp}
We have
\[
T^h = \sum_{\alpha\in \NN^n} h^\alpha H^{(\alpha)}.
\]
In particular, if $\text{char}(\FF)=0$, or $\text{char}(\FF)=p$ and we view the $T^h$ as operators on any subspace of $P_n^{p-1,\dots,p-1}$, we have the relation
\[
    T^h=\prod_{i=1}^n \sum_{\alpha_i\geq 0} \frac{1}{\alpha_i!}h_i^{\alpha_i}\partial_i^{\alpha_i}.
\]
\end{claim}
We can write the last expression as $\exp\left(\sum_{i=1}^n h_i \partial_i \right)=\exp(h\cdot \partial)$ for short; that is, on, say, the space $P_{n}^{p-1,\dots,p-1}$, the shift operators can be thought of as the \emph{exponentials} of the corresponding differential operators in the standard sense for linear operators. Note that when $\text{char}(\FF)=p\neq 0$, in the last expression each sum is understood to stop at $\alpha_i=p-1$, since when $\alpha_i\geq p$ on this space we have $H^{(\alpha)}=0$ and $\partial_i^{\alpha_i}=0$.

\begin{proof}
By the definition of the Hasse derivatives, we have
\[
T^h f(X)=f(X+h)=\sum_{\alpha \in \NN^n} ([Z^\alpha]f(X+Z))h^\alpha = \sum_{\alpha \in \NN^n} h^\alpha H^{(\alpha)}f(X),
\]
as claimed. When $\chr(\FF)=0$, the second part of the Claim follows immediately from~\eqref{eqn:hasse} and factoring; likewise, when $\chr(\FF)=p>0$ and $\deg_i f\leq p-1$ for all $i\in [1,n]$, using \eqref{eqn:hasse} yields
\begin{align*}
    T^h f(X)&=\sum_{\forall i\in [1,n]\: 0\leq \alpha_i \leq p-1} h^\alpha H^{(\alpha)}f(X)\\
    &=\sum_{\forall i\in [1,n]\: 0\leq \alpha_i \leq p-1} \left(\prod_{i=1}^n \frac{h^{\alpha_i}\partial^{\alpha_i}}{\alpha_i !}\right)f(X) \\
    &=\prod_{i=1}^n \sum_{\alpha_i=0}^{p-1} \frac{1}{\alpha_i!}h_i^{\alpha_i}\partial_i^{\alpha_i}f(X),
\end{align*}
as required.
\end{proof}

This relationship between shift operators and derivatives is essential to the application of our methods. To illustrate why, we need to make a few more definitions. 
Given a set $A\subseteq \FF^n$, let $\Lambda_A$ denote the space of linear combinations of $\{T^a\}_{a\in A}$, as operators on $\FF[X_1,\dots,X_n]$. 
Applying Claim~\ref{claim:exp}, each such linear combination $\ell$ can be written as a linear combination of (Hasse) derivatives. In analogy with coefficients of polynomials, we can define $[H^{(\alpha)}]\ell$ as the coefficient of $H^{(\alpha)}$ in $\ell$ when represented this way. Define the degree $\deg (\ell)$ to be the minimal weight over all $\alpha\in \NN^n$ such that $[H^{(\alpha)}]\ell\neq 0$. If such an $\alpha$ does not exist, i.e.~if $\ell$ is identically zero, we write $\deg(\ell)=\infty$. Write $\ell_{(d)}$ for the degree $d$ component of $\ell$ in such a representation; that is,
\[
\ell_{(d)}=\sum_{\alpha:\:|\alpha|=d}([H^{(\alpha)}]\ell) H^{(\alpha)}.
\]
In many cases, it will be helpful to focus on the ``leading component'' $\ell_{(\deg(\ell))}$. Let $\delta(\ell)$ denote this leading component. For each $d\geq 0$, define $\Delta_A^d=\{\ell_{(d)}:\: \ell\in \Lambda_A,\: \deg(\ell)\geq d \}$, and let $\Delta_A=\bigcup_{d\geq 0} \Delta_A^d$. Thus each $\Delta_A^d$ is a space of linear operators on $\FF[X_1,\dots,X_n]$, and $\Delta_A$, the set of all possible leading terms, is a union of a chain of these spaces.

The definitions above can be generalized to work with $A$ replaced by a multiset. Given a set $A\subseteq \FF^n$ and a function $m:A\to \NN$, we denote by $(A,m)$ the multiset containing each $a\in A$ exactly $m(a)$ times. Define the size of $(A,m)$ by $|(A,m)|=\sum_{a\in A} m(a)$. For $\beta\in\NN^n$, we write $(a,\beta)\in (A,m)$ to mean $a\in A$ and $0\leq |\beta| \leq m(a)-1$. For $\beta\geq 0$, we can define the \emph{multishift operators}
\[
(T^h)^{(\beta)}=\sum_{\alpha\in \NN^n}(H^{(\beta)}(h^\alpha))H^{(\alpha)}.
\]
Here the $H^{(\beta)}$ is being applied to $h^\alpha$ as a polynomial in $h$. Note that we have
\begin{align*}
    (T^h)^{(\beta)}(f)&=\sum_{\alpha\in \NN^n}([Y^\beta](h+Y)^\alpha)H^{(\alpha)}(f)\\
    &=[Y^\beta]\sum_{\alpha\in \NN^n}((h+Y)^\alpha)[Z^\alpha]f(X+Z)=[Y^\beta]f(X+h+Y)\\
    &=H^{(\beta)}T^h f=T^h H^{(\beta)}f,
\end{align*}
so that the operator $(T^h)^{(\beta)}$ is essentially a shift operator composed with a derivative. Then, we can define $\Lambda_{(A,m)}$ analogously to before as the set of linear combinations of $\{(T^a)^{(\beta)}\}_{(a,\beta)\in (A,m)}$. The definitions of $[H^{(\alpha)}]\ell, \deg(\ell), \ell_{(d)}, \delta(\ell), \Delta_A^d,$ and $\Delta_A$ then extend readily when $A$ is replaced by $(A,m)$ with $\ell\in \Lambda_{(A,m)}$.

\section{Shift Operators in One Dimension}

\label{sec:shift1d}

To build up some intuition behind these definitions, we will start off by demonstrating some useful properties of shift operators in the case $n=1$. 
From Claim~\ref{claim:exp} we have
\[
T^h=\sum_{k\geq 0}h^k H^{(k)},
\]
for any $h\in \FF$. Fix a subset $A\subseteq \FF$, so that $\Lambda_A=\{\sum_{a\in A} c_a T^a:\: c_a \in \FF\: \forall a\in A\}$. In this one-dimensional case, $\delta(\ell)=c H^{(\deg (\ell))}$ for some $c\in \FF\setminus \{0\}$, so that for each $d\geq 0$, $\Delta_A^d$ is either $\langle H^{(d)} \rangle$ or $\{0\}$. Thus, $\Delta_A$ is completely determined by the range of values attainable by $\deg(\ell)$ for $\ell\in \Lambda_A$, which is characterized by the following result.

\begin{prop}[One-dimensional degree lemma]
\label{prop:1ddl}
Let $A$ be a subset of $\FF$. Then,
\begin{enumerate}
    \item[(i)] Every $\ell\in \Lambda_A$ with at least one nonzero coefficient satisfies $\deg(\ell)\leq |A|-1$, and
    \item[(ii)] For each integer $d\in [0,|A|-1]$, there exists such $\ell\in \Lambda_A$ achieving $\deg(\ell)=d$.
\end{enumerate}
\end{prop}

\begin{proof}
Item~(i) can be checked via any of several simple computations that boil down to evaluating a Vandermonde determinant (as in our proof of the Cauchy-Davenport Theorem in Section~\ref{sec:cd}), essentially showing that the vectors given by the first $|A|$ coefficients of each $T^a$ are linearly independent. We include a more conceptual argument here for the sake of completeness and illustration.

For each $a\in A$, define $v_a=(a^0,\dots,a^{|A|-1})\in \FF^{|A|}$. It suffices to show that $\{v_a\}_{a\in A}$ is linearly independent, i.e.~that the $|A|\times |A|$ matrix $M$ whose rows are the $v_a$ is nonsingular. Indeed, suppose we have constants $c_0,\dots,c_{|A|-1}$ such that for all $a\in A$ we have
\[
\sum_{k=0}^{|A|-1} c_k a^k = 0.
\]
Then the polynomial $P(X)=\sum_{k=0}^{|A|-1} c_k X^k$ is a polynomial of degree at most $|A|-1$ that vanishes on the set $|A|$, meaning $P(X)$ is identically zero, and thus $c_0=\cdots=c_{|A|-1}=0$, so that $M$ is nonsingular as needed, proving item~(i).

Applying item~(i) to a subset $A'\subseteq A$ of size $d+1$ yields that $\deg(\ell)\leq d$ for all $\ell\in \Lambda_{A'}$. On the other hand, item~(i) implies in particular that the set of operators $\{T^a\}_{a\in A'}$ is linearly independent, so that $\dim(\Lambda_{A'})=|A'|=d+1$, while the subspace $\{\ell\in \Lambda_{A'}\mid \: [H^{(k)}]\ell=0\:\forall k\in [0,d-1]\}$ has codimension at most $d$. Thus, $\deg(\ell)=d$ for some $\ell\in \Lambda_{A'}\subseteq \Lambda_A$, as needed.
\end{proof}

From the definition of $\deg(\ell)$ and the properties of Hasse derivatives, we know that for any polynomial $f(X)$,
\[
\deg(\ell(f))\leq \deg(f)-\deg(\ell).
\]
Equality always holds when $\chr(\FF)=0$. When $\chr(\FF)=p>0$, equality occurs when $\binom{\deg(f)}{\deg(\ell)}\not\equiv 0 \pmod p$, which holds, for example, if $\deg(\ell)\leq \deg(f)<p$. Proposition~\ref{prop:1ddl} then immediately implies the following.

\begin{cor}
\label{cor:shifts}
Let $f(X)\in \FF[X]$ be a polynomial with $\deg(f)=D$, where $\chr(\FF)=0$ or $D<\chr(\FF)$. For any set $A\subset \FF$ with $|A|\leq D+1$, the set of polynomials
\[
\{T^{a}f\}_{a\in A}
\]
is linearly independent. In particular, there is a linear combination of the elements of $\{T^{a}f\}_{a\in A}$ whose degree is exactly $D-(|A|-1)$.
\end{cor}

A generalization of Proposition~\ref{prop:1ddl} for multisets can be shown with a very similar proof. The details in this proof around the relationship between derivatives and multiplicity explain the motivation behind the definition we chose for $\Lambda_{(A,m)}$.
\begin{prop}[One-dimensional multidegree lemma]
\label{prop:1dml}
Let $(A,m)$ be a multiset in $\FF$. Then,
\begin{enumerate}
    \item[(i)] Every $\ell\in \Lambda_{(A,m)}$ with at least one nonzero coefficient satisfies $\deg(\ell)\leq |(A,m)|-1$, and
    \item[(ii)] For each integer $d\in [0,|(A,m)|-1]$, there exists such $\ell\in \Lambda_A$ achieving $\deg(\ell)=d$.
\end{enumerate}
\end{prop}

\begin{proof}
As before, item~(i) applied to a submultiset of $(A,m)$ of the appropriate size implies the operators $(T^{a})^{j}$ are all linearly independent, which yields item~(ii) by dimension counting. It remains to prove item~(i). Let $N=|(A,m)|$. For each ordered pair $(a,j)\in (A,m)$, define $v_{a,j}=(H^{(j)}(a^0),\dots,H^{(j)}(a^{N-1}))\in \FF^{N}$. It suffices to show that $\{v_{a,j}\}_{(a,j)\in (A,m)}$ is linearly independent, i.e.~that the $N\times N$ matrix $M$ whose rows are the $v_{a,j}$ is nonsingular. Indeed, suppose we have constants $c_0,\dots,c_{N-1}$ such that for all $(a,j)\in (A,m)$ we have
\[
\sum_{k=0}^{N-1} c_k H^{(j)}(a^k) = 0.
\]
Let $P(X)=\sum_{k=0}^{N-1} c_k X^k$, so that $(H^{(j)}P)(a)=0$ whenever $(a,j)\in (A,m)$. Then $\prod_{a\in A}(X-a)^{m(a)}|P(X)$. But $\deg P\leq N-1<\sum_{a\in A}m(a)$, so $P(X)$ must be identically zero, and thus $c_0=\cdots=c_{N-1}=0$, so that $M$ is nonsingular as needed.
\end{proof}

The following result gives an explicit way to go from a linear combination $\ell$ of (multi)shift operators with $\deg(\ell)=d$ to a linear combination $\tilde \ell$ with $\deg(\tilde \ell)=d-1$. This is rarely useful in one dimension because Proposition~\ref{prop:1dml} gives an exact characterization of possible leading terms, but will help in proving its more useful analogue in multiple dimensions.
\begin{lem}[One-dimensional reduction lemma]
\label{lem:1drl}
Given a multiset $(A,m)\subset \FF$ and a linear combination $\ell\in \Lambda_{(A,m)}$, suppose
\[
\ell=\sum_{k\geq 0} C_k H^{(k)}.
\]
Then there exists a linear combination $\tilde \ell\in \Lambda_{(A,m)}$ that can be expanded as
\[
\tilde \ell = \sum_{k\geq 0} C_{k+1} H^{(k)}.
\]
\end{lem}

\begin{proof}
Given
\[
\ell=\sum_{(a,j)\in (A,m)} c_{a,j}(T^a)^{(j)},
\]
let
\[
\tilde \ell = \sum_{(a,j)\in (A,m)} (c_{a,j}a+c_{a,j+1})(T^a)^{(j)},
\]
where we set $c_{a,j+1}=0$ if $(a,j+1)\notin (A,m)$.

Then for any $k\geq 0$ we have
\begin{align*}
    [H^{(k)}]\tilde \ell &= \sum_{(a,j)\in (A,m)} (c_{a,j}a+c_{a,j+1})(H^{(j)}(a^k)) \\
    &= \sum_{(a,j)\in (A,m)} c_{a,j}(a H^{(j)}(a^k) + H^{(j-1)}(a^k)) \\ 
    &= \sum_{(a,j)\in (A,m)} c_{a,j} a^{k-j+1}\left(\binom{k}{j} + \binom{k}{j-1}\right) \\
    &= \sum_{(a,j)\in (A,m)} c_{a,j} a^{k-j+1}\binom{k+1}{j} = [H^{(k+1)}]\ell,
\end{align*}
as claimed.
\end{proof}
\begin{rem}
From the proof above we can also conclude the following: for any $a\in A$, $\tilde \ell_a:= \tilde \ell - a \ell$ is a linear combination satisfying $\delta(\tilde \ell_a)=\delta(\tilde \ell)$, but the coefficient of $(T^{a})^{(m(a)-1)}$ in $\tilde \ell_a$ is $c_{a,m(a)-1}a-c_{a,m(a)-1}a=0$. Thus, we can obtain a linear combination of the desired leading term with any one point $\tilde a$ of our choice reduced in multiplicity by one. This observation, again, only becomes useful in multiple dimensions.
\end{rem}

\section{Applications of the One-Dimensional Shift Operator Method}
\label{sec:apply1d}
Before moving on to our discussion of shift operators in a multi-dimensional setting, we present quick proofs of a few established results using the ideas in the previous section.

\subsection{The Combinatorial Nullstellensatz}

\label{subsec:cns}

\hspace{10pt}

We begin with a simple proof of the nonvanishing lemma form of Alon's Combinatorial Nullstellensatz \cite{alonCNS}. It is essentially equivalent to the proof given in \cite{tao2013survey}, though our formulation in terms of shift operators seems more prone to generalization. We in fact prove a generalization of the nonvanishing lemma as stated in \cite{BatzayaBayarmagnai2020CNSgen} -- as described in the introduction, the maximum degree condition is replaced by a maximal monomial condition, and the sets $A_i$ are replaced with multisets. Recall that $x^\alpha$ is a maximal monomial in a polynomial $f$ if $[x^\alpha]f\neq 0$ but $[x^\beta]f=0$ for all $\beta > \alpha$.

\begin{theorem}[Generalized Combinatorial Nullstellensatz, {\cite[Corollary~1.6]{BatzayaBayarmagnai2020CNSgen}}]
\label{thm:CNSnonvanish}
Let $\FF$ be an arbitrary field, and let $f\in \FF[X_1,\dots,X_n]$. Suppose $X^\alpha$ is a monomial with nonzero coefficient in $f$ and $\alpha$ is a maximal monomial in $f$. Then, given multisets $(A_1,m_1),\dots,(A_n,m_n)\subset \FF$ satisfying $|(A_i,m_i)|\geq \alpha_i+1$ for all $i\in [1,n]$, there are $(a_1,r_1)\in (A_1,m_1),\dots,(a_n,r_n)\in (A_n,m_n)$ so that $H^{(r_1,\dots,r_n)}f(a_1,\dots,a_n)\neq 0$.
\end{theorem}
\begin{proof}
For $1\leq i\leq n$, let $e_i\in \ZZ^n$ denote the $n$-tuple with a $1$ in the $i$th coordinate and $0$s everywhere else, so that $(\alpha_1,\dots,\alpha_n)=\sum_{i=1}^n \alpha_i e_i$. For convenience, we write $T_i^h=T^{h e_i}$ and $H_i^{(j)}=H^{(j e_i)}$. For $1\leq i\leq n$, since $|(A_i,m_i)|\geq \alpha_i+1$, by Proposition~\ref{prop:1dml} we have a linear combination $\ell_i\in \Lambda_{(A_i,m_i)}$ such that $\deg(\ell_i)=\alpha_i$. Then $\ell_i X_i^{d_i}=0$ for all $d_i<\alpha_i$, but $\ell_i X_i^{\alpha_i}$ is a nonzero constant. 
Consider the product $\prod_{i=1}^n \ell_i$ of these operators. For any $\beta\in \NN^n$, we have
\[
\left(\prod_{i=1}^n \ell_i\right) X^\beta = \prod_{i=1}^n (\ell_i X_i^{\beta_i}).
\]
By the maximality of $X^\alpha$ in $f$, for every monomial $X^\beta$ in $f$ with $\beta\neq \alpha$ there is some $i$ such that $\beta_i<\alpha_i$, so that $\ell_i X_i^{\beta_i}=0$. Thus,
\[
\left(\prod_{i=1}^n \ell_i\right) f = \left([X^\alpha]f\right) \left(\prod_{i=1}^n \ell_i\right) X^\alpha = \left([X^\alpha]f\right) \left(\prod_{i=1}^n \ell_i X_i^{\alpha_i} \right).
\]
The right hand side is a nonzero constant. Expanding out the left hand side and evaluating at $X=0$ yields a linear combination of terms of the form
\[
\left(\left(\prod_{i=1}^n (T_i^{a_i})^{(r_i)}\right)f\right)(0)=\left(\left(\prod_{i=1}^n T_i^{a_i} H_i^{(r_i)}\right)f\right)(0)=H^{(r_1,\dots,r_n)}f(a_1,\dots,a_n).
\]
Thus, one of these terms must be nonzero, as desired.
\end{proof}

Notice that while the proof above deals with multishift operators in $n$ dimensions, it separates them into products of $n$ one-dimensional multishift operators, so that only the one-dimensional versions of the properties we established are needed, since the operators on different variables do not end up interacting with each other. It is natural to ask what we can do when we replace this box-like view of $n$-dimensional point sets with a more ball-like view that does not treat the $n$ coordinate directions differently from any other directions; this will be explored once we have established the properties of shift operators in multiple dimensions, in Sections~\ref{sec:shiftmd} and~\ref{sec:applymd}.

This proof connects our techniques directly to the usual polynomial method based on the Combinatorial Nullstellensatz, since we can rederive most of the results of the classical method through the nonvanishing lemma. On the other hand, the next result suggests that our variant of the method can be applied somewhat more broadly, even in one dimension, since it gives us more room to work with the concept of divisibility and factors of high multiplicity.

\subsection{Sumsets Constrained by Lacunary Polynomials}
\label{sec:lacu}
\hspace{10pt}

For a fixed prime $p$ and a positive integer $d|p-1$, let $Z_d \subseteq \FF_p$ denote the multiplicative subgroup of $\FF_p^\times$ consisting of the elements whose order is a divisor of $d$. In \cite{hp}, Hanson and Petridis show the following result using Stepanov's method, a variant of the polynomial method that takes advantage of factors of high multiplicity:
\begin{theorem}[Hanson, Petridis, {\cite[Theorem~1.2]{hp}}]
\label{thm:hp}
Let $p$ be a prime and suppose $A,B\subseteq \FF_p$ satisfy $A+B\subseteq Z_d\cup \{0\}$ for some $d$ properly dividing $p-1$. Then
\[
|A||B|\leq d+|B\cap (-A)|.
\]
\end{theorem}

Here we present a version of their proof in the language of shift operators; the translation into our framework is very natural, suggesting a strong connection between their methods and ours. 
The actual result we prove is slightly more general; the statement we give here is technical, but designed to highlight the important ideas in the proof.

\begin{theorem}
\label{thm:hptechnical}
Let $p$ be a prime, and let $A,B\subseteq \FF_p$. Suppose for some integers $r \leq \min(|B|-2,d)$ and $d<p$ there exists a polynomial $F(z)\in \FF_p[z]$ of the form
\[
F(z)=z^d+R(z),
\]
where $\deg(R)=r$, such that $zF(z)$ vanishes on $A+B$. Then
\[
|A|(|B|-r)\leq d-r+|B\cup (-A)|.
\]
\end{theorem}
Note that setting $F(z)=z^d-1$ recovers Theorem~\ref{thm:hp}. The fact that this result hinges on the second highest degree term of $F$ being of low degree suggests connections with the study of \emph{lacunary polynomials}; it may be fruitful to compare this approach with the method employed by Di Benedetto, Solymosi, and White in \cite{dbsw} to obtain a result very closely related to that of \cite{hp}.
\begin{proof}
Let
\[
f(z)=\prod_{c\in A+B} (z-c), \qquad g(x)=\prod_{a\in A}(x-a).
\]
For convenience, define $g_0(x)=\prod_{a\in A\cap (-B)}(x-a)$ and $g_1(x)=\prod_{a\in A\setminus (-B)}(x-a)$, and let $K=|B|-1-r$. Define $F_k(z)=z^k F(z)$ for $0\leq k\leq K$, so $f|F_k$ when $k\geq 1$. By Proposition~\ref{prop:1ddl}, we can find $\ell\in \Lambda_B$ such that $\deg(\ell)=|B|-1$. Since $g|T^b f$ for any $b\in B$, we have $g|T^b F_k$ for all $b\in B$ and $k\geq 1$, and $g_1|T^b F_0$ for all $b\in B$. Thus, $g_0(x)|\ell(F_k)$ for $1\leq k\leq K$ and $g_1(x)|\ell(F_k)$ for $0\leq k\leq K$. However, for $0\leq k\leq K-1$, we have
\begin{align*}
    \ell(F_k)=\ell(z^{d+k})+\ell(z^k R(z))=\ell(z^{d+k}),
\end{align*}
since $\deg(z^k R(z))\leq |B|-2$. Thus, for each $k<K$ we have
\begin{align*}
    H^{(K-k)}\ell(F_K) &=H^{(K-k)}\ell(z^{d+K})+H^{(K-k)}\ell(z^K R(z))\\
    &=\ell(H^{(K-k)}z^{d+K})+0=\varepsilon_{d+K,K-k} \ell(z^{d+k})\\
    &=\varepsilon_{d+K,K-k} \ell(F_k),
\end{align*}
where $\varepsilon_{d+K,K-k}=\binom{d+K}{K-k}$ is a constant. 
Here we have used the fact that $(K-k)+\deg(\ell)\ge 1+|B|-1>K+r=\deg(z^k R(z))$, as well as the fact that Hasse derivative operators commute with each other (and thus with linear combinations of shift operators). Then $g_0(x)|H^{(j)} \ell(F_K)$ for $0\leq j\leq K-1$, while $g_1(x)|H^{(j)} \ell(F_K)$ for $0\leq j\leq K$. Since $g_0,g_1$ are relatively prime and each consists of a product of distinct linear factors, we thus obtain $g_0^{K} g_1^{K+1}|\ell(F_K)$, so
\[
d+K-(|B|-1)=\deg(\ell(F_K))\geq (K+1)\deg g - \deg g_0.
\]
This simplifies to
\[
d-r+|B\cap (-A)|\geq (K+1)\deg g = |A|(|B|-r),
\]
which is the desired result.
\end{proof}

\section{Shift Operators in Multiple Dimensions}
\label{sec:shiftmd}

As promised, we now turn to the study of shift operators in dimensions $n>1$. 
In this setting, the range of values attained by $\deg(\ell)$ for $\ell\in \Lambda_{A}$ no longer completely determines the set of possible leading terms $\Delta_{A}$. Nevertheless, knowing the largest value attainable by $\deg(\ell)$, i.e.~the largest $d$ such that $\Delta_{A}^d\neq \{0\}$, can be valuable. Let $\deg(A)$, or $\deg(A,m)$ for a multiset, denote this largest possible degree. For instance, when $n=1$, Propositions~\ref{prop:1ddl} and~\ref{prop:1dml} yield that $\deg(A)=|A|-1$ and $\deg(A,m)=|(A,m)|-1$. We define $\deg(\emptyset)=-\infty$ for convenience. In general, it is clear that $\deg(A)$ is invariant under invertible affine transformations on $\FF^n$ applied to the set $A$. In particular, if $A$ is contained in a $k$-dimensional flat of $\FF^n$ for some $k<n$, we can change coordinates to analyze $A$ as a subset of a subspace $\FF^k$ to determine $\deg(A)$. 

Arguments similar to those in the proof of Proposition~\ref{prop:1ddl} yield upper and lower bounds on $\deg(A)$ in the multi-dimensional case as well. In this case, however, the bounds do not match, so more must be known about the set $A$ to determine $\deg(A)$ (and more generally, the structure of $\Delta_A^d$) precisely. Here is where a multi-dimensional analogue of the reduction lemma, Lemma~\ref{lem:1drl}, proves to be useful. As in the proof of Theorem~\ref{thm:CNSnonvanish}, we let $e_i\in \ZZ^n$ denote the $n$-tuple with a $1$ in the $i$th coordinate and $0$s everywhere else.

\begin{lem}[Multi-dimensional reduction lemma]
\label{lem:mdrl}
Given a multiset $(A,m)\subset \FF^n$ and a linear combination $\ell\in \Lambda_{(A,m)}$, suppose
\[
\ell=\sum_{\alpha\in \NN^n} C_\alpha H^{(\alpha)}.
\]
Then for each $i\in [1,n]$ there exists a linear combination $\tilde \ell_i\in \Lambda_{(A,m)}$ that can be expanded as
\[
\tilde \ell_i = \sum_{\alpha\in \NN^n} C_{\alpha+e_i} H^{(\alpha)}.
\]
\end{lem}

\begin{proof}
The proof is very similar to the one-dimensional case. Given
\[
\ell=\sum_{(a,\beta)\in (A,m)} c_{a,\beta}(T^a)^{(\beta)},
\]
let
\[
\tilde \ell_i = \sum_{(a,\beta)\in (A,m)} (c_{a,\beta}a_i+c_{a,\beta+e_i})(T^a)^{(\beta)},
\]
where we set $c_{a,\beta+e_i}=0$ if $(a,\beta+e_i)\notin (A,m)$.

Then for any $\alpha\in \NN^n$ we have
\begin{align*}
    [H^{(\alpha)}]\tilde \ell_i &= \sum_{(a,\beta)\in (A,m)} (c_{a,\beta}a_i+c_{a,\beta+e_i})(H^{(\beta)}(a^\alpha)) \\
    &= \sum_{(a,\beta)\in (A,m)} c_{a,\beta}(a_i H^{(\beta)}(a^\alpha) + H^{(\beta-e_i)}(a^\alpha)) \\ 
    &= \sum_{(a,\beta)\in (A,m)} c_{a,\beta} a^{\alpha-\beta+e_i}\left(\binom{\alpha}{\beta} + \binom{\alpha}{\beta-e_i}\right) \\
    &= \sum_{(a,\beta)\in (A,m)} c_{a,\beta} a^{\alpha-\beta+e_i}\binom{\alpha+e_i}{\beta} = [H^{(\alpha+e_i)}]\ell,
\end{align*}
as claimed.
\end{proof}
Much of the importance of this lemma comes from the multidimensional analogue of the remark under the proof of Lemma~\ref{lem:1drl}. Namely, the linear combination $\tilde\ell_{i,\varepsilon}:=\tilde \ell_i - \varepsilon \ell \in \Lambda_{(A,m)}$ satisfies $\delta(\tilde\ell_{i,\varepsilon})=\delta(\tilde\ell_i)$, but the coefficient of $(T^a)^\beta$ in $\tilde\ell_{i,\varepsilon}$ is $0$ for all $(a,\beta)$ such that $a_i=\varepsilon$ and $c_{a,\beta+e_i}=0$ (in particular, all $\beta$ of weight $m(a)-1$ for such $a$). Therefore, we have the following result on the structure of $\Delta_{(A,m)}$.

\begin{cor}
\label{cor:mdrl}
Let $(A,m)\subset \FF^n$ with $0\neq \sum_{|\alpha|=d} C_\alpha H^{(\alpha)}\in \Delta_{(A,m)}^d$. For $1\leq i\leq n$, and for any affine hyperplane $S\subset \FF^n$ parallel to $\langle e_1,\dots,e_{i-1},e_{i+1},\dots,e_n\rangle$, define $m_{S}:A\to \NN$ by
\[
m_S(a)=\begin{cases}m(a)\text{ if }a\notin S, \\ \max(m(a)-1,0)\text{ if }a\in S.\end{cases}
\]
Then
\begin{itemize}
    \item[(i)] $\sum_{|\alpha|=d,\: \alpha_i>0} C_\alpha H^{(\alpha-e_i)}\in \Delta_{(A,m_S)}^{d-1},$ and
    \item[(ii)] If $\sum_{|\alpha|=d,\: \alpha_i>0} C_\alpha H^{(\alpha)}=0$, then $\deg(A,m_S)\geq d$.
\end{itemize}
\end{cor}
That is, if a given leading term $\delta(\ell)$ is attained by a linear combination $\ell$ from a multiset $(A,m)$, then we can also attain the leading term of one smaller degree obtained by ``differentiating'' $\delta(\ell)$ with respect to any direction $i$. Furthermore, for any hyperplane $S$ ``orthogonal'' to $e_i$, this new leading term can be obtained from the same multiset, but with the multiplicity of each point on $S$ decreased by one.

Let us give a simple example of an application of this corollary to illustrate the idea. Let $\FF^n=\RR^2$ be the Cartesian plane; since $\RR$ has characteristic zero, we can write ordinary derivatives instead of Hasse derivatives. Let $A\subset \RR^2$ be a finite set of points. Let $\ell\in \Lambda_A$ be some linear combination of the shift operators $\{T^a\}_{a\in A}$, with $\deg(\ell)=d$. Then the lowest degree term $\delta(\ell)$ is $R(\partial_x,\partial_y)$, for some homogeneous polynomial $R\in \RR[Z_1,Z_2]$ of degree $d$. Corollary~\ref{cor:mdrl} then tells us that, even after removing any one vertical line from $A$, we can still find a linear combination of the remaining shift operators with lowest degree term $\partial_{Z_1}R(\partial_x,\partial_y)$. A similar reduction holds for lines of any other direction, with $\partial_{Z_1}$ replaced by the corresponding directional derivative (or after a change of basis). By iterating Corollary~\ref{cor:mdrl}, we can obtain $\partial^\alpha R(\partial_x,\partial_y)\in \Delta_A$ for any $\alpha\in \NN^2$. (Of course, only finitely many of these are nonzero for a given $R$.)

The construction used in the above corollary gives a simple proof of the following basic but important fact.

\begin{lemma}
\label{lem:mlindep}
For any multiset $(A,m)\subset \FF^n$, the set $\{(T^a)^{(\beta)}\}_{(a,\beta)\in(A,m)}$ is linearly independent.
\end{lemma}
\begin{proof}
Suppose the claim is false, and take a counterexample with $n$ minimal, and with $|(A,m)|$ minimal subject to that value of $n$. Then for some choices of $c_{a,\beta}$, not all zero, we have $\ell:=\sum_{(a,\beta)\in (A,m)}c_{a,\beta}(T^a)^{(\beta)}=0$. Fix a choice of $(a,\beta)$ with $c_{a,\beta}\neq 0$. By the minimality of $|(A,m)|$, we can choose such a pair with $|\beta|=m(a)-1$.

For any $i\in [1,n]$, applying the construction in Lemma~\ref{lem:mdrl} and Corollary~\ref{cor:mdrl} yields, for some hyperplane $S\subset \FF^n$ containing $a$, a linear combination $\tilde \ell_{i,a_i} := \tilde \ell_i - a_i \ell\in \Lambda_{(A,m_S)}$ that equals zero. Since $|(A,m_S)|<|(A,m)|$, this is a contradiction unless all coefficients in $\tilde \ell_{i,a_i}$ are zero. But for any $(a',\beta')$, the coefficient of $(T^{a'})^{(\beta')}$ in $\tilde\ell_{i,a_i}$ is $c_{a',\beta'}(a_i'-a_i)+c_{a',\beta'+e_i}$. If there is a second point $a'\neq a$ in $(A,m)$, again by the minimality of $|(A,m)|$ we can choose $\beta'$ such that $|\beta'|=m(a')-1$ and $c_{a',\beta'}\neq 0$. Then, choosing $i$ such that $a_i\neq a_i'$, the coefficient of $(T^{a'})^{(\beta')}$ in $\tilde\ell_{i,a_i}$ is $c_{a',\beta'}(a_i'-a_i)+c_{a',\beta'+e_i}=c_{a',\beta'}(a_i'-a_i)\neq 0$.

Otherwise, $a$ is (up to repetition) the only point in $(A,m)$, so in order to have $\ell=0$, we need $m(a)>1$, so $|\beta|\geq 1$. Then, choosing $i$ such that $\beta_i>0$, the coefficient of $(T^{a})^{(\beta-e_i)}$ in $\tilde\ell_{i,a_i}$ is $c_{a,(\beta-e_i)+e_i}\neq 0$. Thus, in either case, $\tilde\ell_{i,a_i}$ has a nonzero coefficient, and we have arrived at a counterexample with a smaller value of $|(A,m)|$, contradicting our assumption of minimality. Thus no such counterexample exists.
\end{proof}
Lemma~\ref{lem:mlindep} implies that, for any $(A,m)\subset \FF^n$, we have 
\[
\dim(\Lambda_{(A,m)})=\sum_{a\in A} \binom{m(a)+n-1}{n}.
\]

We are now ready to prove an inductive analogue of Proposition~\ref{prop:1ddl}. Given a subspace $W\leq \FF^n$, fix a subspace $W^\perp \leq \FF^n$ such that $\FF^n=W\oplus W^\perp$. The elements $c\in W^\perp$ index the cosets $c+W$ of $W$ in $\FF^n$. 
For the lower bound on $\deg(A)$, we need the following definition. For $n,r\in \NN$ and $q\in \NN\cup \infty$, define
\[
N(n,q,r)=\left|\left\{(d_1,\dots,d_n)\in \NN^n:\: d_i<q,\, \sum_{i=1}^n d_i\leq r\right\}\right|.
\]
This quantity can be interpreted as the number of monomials in $n$ variables of total degree at most $r$, with degree less than $q$ in each variable. We have the trivial upper bound $N(n,q,r)\leq N(n,\infty,r)=\binom{n+r}{r}$.

\begin{prop}[Multi-dimensional degree lemma]
\label{prop:nddl}
Let $A$ be a nonempty subset of $\FF^n$. For $k\in \NN$, let $C_{W,A,k}\subseteq W^\perp$ be the set of $c\in W^\perp$ such that $\deg ((c+W)\cap A)\geq k$.
\begin{enumerate}
    \item[(i)] We have
    \[
        \deg(A)\leq \min_{W\leq \FF^n} \max_{k:\: |C_{W,A,k}|>0} (\deg(C_{W,A,k})+k).
    \]
    \item[(ii)] If for some subspace $W\leq \FF^n$ and integer $r\geq 0$ we have $|A\cap W|>N(\dim(W),|\FF|,r)$, then $\deg(A)>r$.
\end{enumerate}
\end{prop}

\begin{proof}
(i) Given $A\subseteq\FF^n$, we wish to show that for every subspace $W\leq \FF^n$ there is some $k\geq 0$ such that $\deg(A)\leq \deg(C_{W,A,k})+k$. Fix such a subspace $W$, and fix a nonzero linear combination $\ell\in \Lambda_A$. Since $\deg(A)$ is invariant under changes of coordinates, without loss of generality let $W$ be the subspace spanned by the first $d$ coordinates for some $d\leq n$, and take $W^\perp$ to be the subspace spanned by the remaining $n-d$ coordinates. If $d=n$, the statement is vacuous with $k=\deg(A)$, so assume $d<n$.

For every $c\in W^\perp$, let $\ell_c$ be the restriction of $\ell$ to the coset $c+W$; that is, the result of keeping only the terms $T^h$ with $h\in c+W$ in the linear combination $\ell$. Let $\tilde \ell_c=\ell_c T^{-c}$, so $\tilde \ell_c\in \Lambda_{W}$, and thus $[H^{(\alpha)}]\tilde \ell_c=0$ unless $\alpha$ is supported on its first $d$ coordinates.

Since $\ell\neq 0$, we can find some $\alpha_0\in \NN^d$ with $k:=|\alpha_0|$ minimal such that $[H^{(\alpha_0)}]\tilde\ell_{c_0}\neq 0$ for some $c_0\in W^\perp$. Then $k\leq \deg((c_0+W)\cap A)$. By the minimality of $k$, we must have $\tilde \ell_{c}=0$ for all $c\notin C_{W,A,k}$. Then
\begin{align*}
    \ell &= \sum_{c\in C_{W,A,k}} \tilde\ell_c T^c \\
    &= \sum_{c\in C_{W,A,k}} \sum_{\alpha\in \NN^d} ([H^{(\alpha)}] \tilde\ell_c) T^c H^{\alpha} \\
    &= \sum_{\alpha\in \NN^d} H^{\alpha} \sum_{c\in C_{W,A,k}} ([H^{(\alpha)}] \tilde\ell_c) T^c.
\end{align*}
The inner sum $\sum_{c\in C_{W,A,k}} ([H^{(\alpha)}] \tilde\ell_c) T^c$ is a linear combination of the operators $T^c$ for $c\in C_{W,A,k}\subseteq W^\perp$, and can therefore be written as a linear combination of terms $H^{(\beta)}$ for $\beta\in \NN^n$ supported on the last $n-d$ coordinates. By definition, for each $\alpha$ where at least one of the coefficients $[H^{(\alpha)}] \tilde\ell_c$ is nonzero, and in particular for $\alpha_0$, there will be some such $\beta_0$ with $|\beta_0|\leq \deg(C_{W,A,k})$ such that
\[
[H^{(\beta_0)}]\left(\sum_{c\in C_{W,A,k}} ([H^{(\alpha_0)}] \tilde\ell_c) T^c\right) \neq 0.
\]
When $\alpha$ and $\beta$ are supported on disjoint sets of coordinates, we have $H^{(\alpha)}H^{(\beta)}=H^{(\alpha+\beta)}$. Thus,
\[
[H^{(\alpha_0+\beta_0)}]\ell = [H^{(\beta_0)}]\left(\sum_{c\in C_{W,A,k}} ([H^{(\alpha_0)}] \tilde\ell_c) T^c\right) \neq 0,
\]
so that $\deg(\ell)\leq |\alpha_0+\beta_0|\leq k+\deg(C_{W,A,k})$. Since this holds for all $\ell\in\Lambda_A$, we have $\deg(A)\leq \deg(C_{W,A,k})+k$ as desired.

(ii) Again, we can without loss of generality change our coordinates to assume $W$ is spanned by the first $m:=\dim(W)$ coordinates. The result then follows from a simple dimension-counting argument: The elements of $\Delta_{A\cap W}$ with degree at most $r$ are contained in the span of ``monomials'' $H^{(\alpha)}$ with $\alpha\in \NN^m$, $|\alpha|\leq r$. Further, if $|\FF|<\infty$, we claim that only terms $H^{(\alpha)}$ where $\alpha$ satisfies $\deg_i(\alpha)<|\FF|$ appear in an element of $\Delta_{A\cap W}$. Indeed, by Lemma~\ref{lem:mlindep}, we have $\sum_{d\geq 0}\dim(\Delta_{\FF^n}^d)=|\FF|^n$, and there are $|\FF|^n$ linearly independent possible leading terms $H^{(\alpha)}$ for $\alpha\in [0,|\FF|-1]^n$, each attained by some linear combination $\ell\in \Lambda_B$ for a box $B=B_1\times\cdots\times B_n$ with $|B_i|=\alpha_i+1$, so no other values of $\alpha$ can appear. 
Thus, the elements of $\Delta_{A\cap W}$ with degree at most $r$ are contained in the span of most $N(\dim(W),|\FF|,r)$ elements.

Again by Lemma~\ref{lem:mlindep}, we have $\sum_{d\geq 0}\dim(\Delta_{A\cap W}^d)=|A\cap W|>N(\dim(W),|\FF|,r)$, so there must be some element $\ell\in \Delta_{A\cap W}$ with $\deg(\ell)>r$, as desired.
\end{proof}

\begin{rem}
    There are several important things to note about Proposition~\ref{prop:nddl}. First, the simplest way to use part~(i) of the Proposition is to apply it iteratively, with subspaces of codimension $1$ at each step. In this case, at each step, $C(W,A,k)$ is a one-dimensional set, so that $\deg(C(W,A,k))=|C(W,A,k)|-1$. 
For example, when $A$ is contained in a box $A_1\times \cdots \times A_n$, we obtain the bound
\[
\deg(A)\leq \sum_{i=1}^n (|A_i|-1), 
\]
which is tight, as seen from a product construction. It is worth noting that a version of Proposition~\ref{prop:nddl}(i) where we restrict to subspaces $W$ of codimension $1$ can be proved by simply iteratively applying the construction in Corollary~\ref{cor:mdrl}, picking one hyperplane $c+W$ with $c\in C_{W,A,k}$ at a time and removing its elements from the linear combination.

Second, in certain important cases, in particular when $\FF=\FF_p$, the quantities $N(n,|\FF|,r)$ that come up in part~(ii) of the Proposition are well-studied. In particular, Lemma~9.2 of \cite{lovaszsauermann2019multicolor} yields, for $k\geq 3$ and $m\geq 2$,
\[
\left|N\left(n,m,\frac{(m-1)n}{k}\right)\right|\leq (\Gamma_{m,k})^n,
\]
where
\[
\Gamma_{m,k}=\min_{0<\gamma<1}\frac{1+\gamma+\cdots+\gamma^{m-1}}{\gamma^{(m-1)/k}},
\]
a quantity satisfying $\Gamma_{m,k}<m$ and $\Gamma_{m,m}=:\gamma_m<4$ \cite{naslund2020egz}. 
This bound becomes relevant when studying the $k$-colored sum-free set problem.
\end{rem}

The following simple observation is a key reason that the set $\Delta_A$ and the constant $\deg(A)$ are valuable invariants in studying problems involving sumsets. Recall that given sets $S$ and $T$, $S\cdot T$ denotes the set of pairwise products of elements in $S$ and $T$.

\begin{lemma}
\label{lem:adddeg}
Given sets $A,B\subseteq \FF^n$, we have
\[
\Delta_A \cdot \Delta_B \subseteq \Delta_{A+B}.
\]
In particular, either we have
\[
\deg(A+B)\geq \deg(A) + \deg(B),
\]
or else all $\ell_A\in \Lambda_A$, $\ell_B\in \Lambda_B$ attaining these maximal degrees satisfy $\ell_A\cdot \ell_B=0$.
\end{lemma}
In effect, the invariant $\deg(A)$ offers a multi-dimensional analogue of set size that behaves well in Cauchy-Davenport-like settings -- in $1$ dimension, the analogue of the alternative condition given is that $|A|+|B|\geq p+2$. Unfortunately, the edge cases in more than one dimension are significantly more complicated; often it will be worth considering the overall structure of $\Delta_A$ rather than just the quantity $\deg(A)$.
\begin{proof}
Given $\delta_A\cdot \delta_B \in \Delta_A\cdot \Delta_B$, let $\ell_A\in \Lambda_A$, $\ell_B\in \Lambda_B$ be such that $\delta(\ell_A)=\delta_A$, $\delta(\ell_B)=\delta_B$. Observe that the product $\ell_A\cdot \ell_B$ is a linear combination of terms of the form $T^{a+b}$ for $a\in A$ and $b\in B$, and is thus contained in $\Lambda_{A+B}$. We have $(\ell_A\cdot\ell_B)_{(\deg(\ell_A)+\deg(\ell_B))}=\delta_A\cdot \delta_B$, and $\deg(\ell_A\cdot\ell_B)\geq \deg(\ell_A)+\deg(\ell_B)$. So, either $\deg(\ell_A\cdot\ell_B)= \deg(\ell_A)+\deg(\ell_B)$, and thus $\delta_A\cdot \delta_B \in \Delta_{A+B}$, or $\deg(\ell_A\cdot\ell_B)> \deg(\ell_A)+\deg(\ell_B)$, so $\delta_A\cdot \delta_B=0\in \Delta_{A+B}$. This shows the first part of the lemma. In particular, applying this argument when $\deg(\ell_A)$ and $\deg(\ell_B)$ are maximal shows the second part of the lemma.
\end{proof}

\section{Applications of the Multidimensional Shift-Operator Method}
\label{sec:applymd}

\subsection{The Croot-Lev-Pach method}

In Section~\ref{sec:cd} and Section~\ref{sec:apply1d}, we drew a connection between the one-dimensional version of our method and the Combinatorial Nullstellensatz, first by highlighting the similarity in some of their applications and later by proving the nonvanishing lemma form of the Nullstellensatz itself using our method. For the multidimensional version of our method, we demonstrate a similar connection to the Croot-Lev-Pach-Ellenberg-Gijswijt polynomial method. The Erd\H{o}s-Ginzburg-Ziv problem itself, a special case of which was an early inspiration for that method, serves as a convenient and enlightening example to illustrate this connection.

Recall that the \emph{Erd\H{o}s-Ginzburg-Ziv constant} for $\FF_p^n$ is defined to be the smallest integer $s=s(\FF_p^n)$ such that every sequence of $s$ (not necessarily distinct) elements of $\FF_p^n$ contains $p$ elements summing to zero. Thus, $s(\FF_p^n)-1$ is the length of the longest sequence that does not contain $p$ elements summing to zero. Following \cite{zakharov2020convexEGZ}, we define a related quantity called the \emph{weak Erd\H{o}s-Ginzburg-Ziv constant}, $\mathfrak{w}(\FF_p^n)$, to be the size of the largest subset $A\subset \FF_p^n$ such that $p$ elements $a_1,\dots,a_p\in A$ sum to $0$ if and only if $a_1=\cdots=a_p$. It is easy to check that $s(\FF_p^n)\ge (p-1)\mathfrak{w}(\FF_p^n)+1$.

When $p=3$, a convenient coincidence occurs: Three distinct elements sum to $0$ in $\FF_3^n$ if and only if they lie on a line, and in particular form an arithmetic progression. Thus, the Erd\H{o}s-Ginzburg-Ziv problem for $\FF_3^n$ coincides with the widely studied problem of bounding the size of sets with no nontrivial $3$-term arithmetic progressions. 
It was in this context that Ellenberg and Gijswijt made their breakthrough using the then-new polynomial method of Croot, Lev, and Pach. Recall that $N(n,q,r)$ is the number of monomials in $n$ variables of total degree at most $r$ and degree less than $q$ in each variable.

\begin{theorem}[\cite{ellenberggijswijt}]\label{thm:EGegz}
    Let $A$ be a subset of $\FF_p^n$ containing no nontrivial $3$-term arithmetic progression. Then $|A|\le 3N(n,p,(p-1)n/3)$.
\end{theorem}
In the case $p=3$, their result implies $\mathfrak{w}(\FF_3^n)=o(2.756^n)$. Since three elements of $\FF_3^n$ summing to zero are either all distinct or all equal, this also implies $s(\FF_3^n)=o(2.756^n)$.

The original proof of Theorem~\ref{thm:EGegz} essentially proceeds as follows: Let $d=2(p-1)n/3$, and let $V\subseteq P_n^d \cap P_n^{p-1,\dots,p-1}$ be the space of polynomials of total degree at most $d$, and degree in each variable at most $p-1$, that vanish on $\FF_p^n\setminus (2\cdot A)$. (Here $2\cdot A=\{2a:\: a\in A\}$.) Pick $f\in V$ with maximal support $|\Supp(f)|\ge \dim V \ge N(n,p,d)-|\FF_p^n\setminus (2\cdot A)|$. By assumption, for $x,y\in A$, we have $x+y\in 2\cdot A$ if and only if $x=y$, so $f(x+y)=0$ for $x\neq y \in A$. Now consider the map $F:A\times A\to \FF_p$ given by $F(x,y)=f(x+y)$. Viewing this map as an $|A|\times |A|$ matrix, it is diagonal, and thus has rank equal to its number of nonzero diagonal entries, which is $|\Supp(f)|$. On the other hand, $F$ agrees as a function with $f(x+y)$ for $f\in P_n^d$, which can be written in the form
\[
f(x+y)=\sum_{\substack{\alpha \in \NN^n,\\ |\alpha|\le d/2,\, \alpha_i<p \forall i\in [n]}} x^\alpha p_\alpha(y)+ \sum_{\substack{\beta \in \NN^n,\\ |\beta|\le d/2,\, \beta_i<p \forall i\in [n]}} y^\beta q_\beta(x),
\]
for some polynomials $p_\alpha$ and $q_\beta$. Each term in the sum corresponds to a rank $1$ matrix, so $\rank(F)\le 2N(n,p,d/2)$. Combining this with the lower bound on $|\Supp(f)|$ and the observation that $p^n=N(n,p,d/2)+N(n,p,(p-1)n-d/2)$ indeed gives
\[
|A|\le 3N(n,p,(p-1)n/3).
\]

Soon after this original proof, Tao \cite{tao2016slicerankblog} presented a more symmetric formulation of the proof over $\FF_3$. He directly considers the function $F:A^3\to \FF_3$ given by
\[
F(x,y,z)=\begin{cases}
    1 & \text{if } x+y+z=0 \\ 0 & \text{otherwise}.
\end{cases}
\]
This can be viewed as a diagonal $3$-tensor with size $|A|\times |A|\times |A|$. To proceed, he defines a notion of rank for tensors called \emph{slice rank}: A nonzero $d$-dimensional tensor $F$ has slice rank $1$ if it can be written in the form $F=f(x_i)g(x_1,\dots,x_{i-1},x_{i+1},\dots,x_d)$. In general, the slice rank of a tensor is the smallest number of tensors of slice rank $1$ that generate it as a linear combination. As in the matrix setting, a key observation is that the slice rank of a diagonal tensor equals its number of nonzero entries, which is $|A|$ in this case. On the other hand, we can write an explicit polynomial agreeing with $F$:
\[
F(x,y,z)=\prod_{i=1}^n (1-(x_i+y_i+z_i)^2).
\]
As before, this right hand side is a sum of terms of the form $x^\alpha g(y,z)$ for $|\alpha|\le \frac{2n}{3},\, \alpha_i < 3=p \: \forall i$, and similar terms with the variables permuted. Each of these terms gives a tensor of slice rank $1$, so we again end up with the desired bound $|A|\le 3N(n,3,2n/3)$.

We will now demonstrate how a proof of Theorem~\ref{thm:EGegz} arises from our shift operator framework. Interestingly, this proof contains elements evocative of both the asymmetric and symmetric versions of the original proof.

\begin{proof}[Proof of Theorem~\ref{thm:EGegz}]
Suppose that $|A|>2N(n,p,(p-1)n/3)$. Fix an arbitrary ordering of $A$, and let $S_+$ be the set of $a\in A$ such that for some $\ell\in \Lambda_A$ with $\deg(\ell)>(p-1)n/3$, $a$ is the largest element of $A$ such that $[T^a]\ell\neq 0$. Likewise, let $S_-$ be the set of $a\in A$ such that for some $\ell\in \Lambda_{-2\cdot A}$ with $\deg(\ell)>(p-1)n/3$, $a$ is the smallest element of $A$ such that $[T^{-2a}]\ell\neq 0$. By Proposition~\ref{prop:nddl}(ii) with $W=\FF_p^n$, any subset $A'\subseteq A$ with $|A'|>N(n,p,(p-1)n/3)$ satisfies $\deg(A')>(p-1)n/3$. This shows that $|A\setminus S_+|\le N(n,p,(p-1)n/3)$, and similarly $|A\setminus S_-|\le N(n,p,(p-1)n/3)$. In particular, $|S_+\cap S_-|\ge 1$. Thus, we can pick $\ell_+\in \Lambda_A, \ell_-\in \Lambda_{-2\cdot A}$ with $\deg(\ell_+),\deg(\ell_-)>(p-1)n/3$ such that for some element $b\in A$, $b$ is the largest element with $[T^b]\ell_+\neq 0$ and the smallest element with $[T^{-2b}]\ell_-\neq 0$.

Now consider the linear combination $\ell=\ell_+^2 \ell_-\in \Lambda_{A+A+(-2\cdot A)}\subseteq \Lambda_{\FF_p^n}$. We have $\deg(\ell)\ge 2\deg(\ell_+)+\deg(\ell_-)> (p-1)n$, implying that $\ell=0$ by Proposition~\ref{prop:nddl}(i) applied to $\FF_p^n$ (see the remark under the proof of that result). On the other hand, by assumption, the only triples $(a,b,c)\in A^3$ with $a+b-2c=0$ are the ones with $a=b=c$. So,
\[
[T^0] \ell = \sum_{a\in A} ([T^a] \ell_+)^2 ([T^{-2a}]\ell_-)=([T^b] \ell_+)^2 ([T^{-2b}]\ell_-) \neq 0,
\]
a contradiction. Thus we indeed must have $|A|\le 2N(n,p,(p-1)n/3)$ as desired.
\end{proof}
Just like the proof in~\cite{ellenberggijswijt}, this proof can be easily modified to show a similar result for sets without nontrivial solutions to $\alpha a_1+\beta a_2+\gamma a_3=0$, for any $\alpha,\beta,\gamma$ summing to zero. The same proof also generalizes to the setting of $k$-colored sum-free sets. Recall that a \emph{$k$-colored sum-free set} in $\FF_p^d$ (sometimes called a \emph{multiplicative $k$-matching}, as in \cite{petrov2016groupring}) is a collection of $k$-tuples $(x_{1,j},\dots,x_{k,j})_{j=1}^M$ such that for all $j_1,\dots,j_k\in [1,M]$, we have
\[
\sum_{i=1}^k x_{i,j_i}=0 \text{ if and only if }j_1=\cdots=j_k.
\]

The following result is the $\FF_p$ case of the best known upper bound on sizes of $k$-colored sum-free sets.
\begin{theorem}[see \cite{lovaszsauermann2019multicolor}]
\label{thm:kcolorsumfree}
For every prime $p$ and every integer $k\geq 3$, the size of any $k$-colored sum-free set in $\FF_p^n$ is at most $(\Gamma_{p,k})^n$, where
\[
\Gamma_{p,k}=\min_{0<\gamma<1} \frac{1+\gamma+\cdots+\gamma^{p-1}}{\gamma^{(p-1)/k}}.
\]
\end{theorem}

It is noted in \cite{lovaszsauermann2019multicolor} that $\Gamma_{p,k}<p$ for all $k$; likewise, \cite{naslund2020egz} notes that $\Gamma_{p,p}=:\gamma_p<4$. Here we present a proof of Theorem~\ref{thm:kcolorsumfree} in the language of shift operators. This proof is essentially equivalent to one suggested in \cite{petrov2016groupring}, but our framework provides a somewhat different perspective.

\begin{proof}
Let $A=\{(x_{1,j},\dots,x_{k,j})_{j=1}^M\}$ be a $k$-colored sum-free set in $\FF_p^n$, and let $A_i=\{x_{i,j}\}_{j=1}^M$ for $1\leq i\leq k$. By construction, each $A_i$ must consist of $M$ distinct elements. As in our proof of Theorem~\ref{thm:EGegz}, for $1\leq i\leq k$, we consider the set $S_{i,+}$ (resp. $S_{i,-}$) of indices $j\in [1,M]$ such that for some $\ell_i\in \Lambda_{A_i}$ with $\deg(\ell_i)> \frac{(p-1)n}{k}$, $j$ is the maximal (resp. minimal) index such that $x_{i,j}$ has a nonzero coefficient in $\ell_i$. By Proposition~\ref{prop:nddl}(ii), any set $B\subseteq \FF^n$ with $|B|\geq N(n,p,r)+1$ satisfies $\deg(B)>r$. Applying this result with $r=\frac{(p-1)n}{k}$ to $\{x_{i,j}\}_{j\notin S_{i,+}}$, we see that
\[
|S_{i,+}|\geq M-N\left(n,p,\frac{(p-1)n}{k}\right),
\]
for $1\leq i\leq k$, and similarly for $|S_{i,-}|$. Then we have
\[
|S_{1,+}\cap\cdots\cap S_{k-1,+}\cap S_{k,-}|\geq M - \sum_{i=1}^{k-1} (M-|S_{i,+}|)-(M-|S_{k,-}|) \geq M- k N\left(n,p,\frac{(p-1)n}{k}\right).
\]
Thus, as long as $M > k N(n,p,\frac{(p-1)n}{k})$, this intersection is nonempty, and we have some linear combinations $\ell_1\in \Lambda_{A_1},\dots,\ell_k\in \Lambda_{A_k}$ such that $\deg(\ell_i)>\frac{(p-1)n}{k}$ for each $i$, and there is a unique index $j$ such that $[T^{x_{i,j}}]\ell_i\neq 0$ for each $i$. Let $\ell=\prod_{i=1}^k \ell_i$. We have
\[
\deg(\ell) \geq \sum_{i=1}^k \deg(\ell_i) > (p-1)n,
\]
implying again that $\ell=0$ by Proposition~\ref{prop:nddl}(i). On the other hand, upon expanding out each $\ell_i$ and multiplying through, we see that $[T^0]\ell$ can only be contributed to by tuples $(x_1,\dots,x_k)\in A_1\times\cdots\times A_k$ with $\sum_{i=1}^k x_i=0$. Since $A$ is a $k$-color sum-free set, by assumption the only such tuples are of the form $(x_{1,j},\dots,x_{k,j})$ for some $j$, but by construction there is exactly one such index $j_0$ such that for each $i$, $[T^{x_{i,j_0}}]\ell_i\neq 0$. Then $[T^0]\ell = \prod_{i=1}^k [T^{x_{i,j_0}}]\ell_i \neq 0$, a contradiction. Hence any $k$-colored sum-free set $A$ must satisfy $|A|\leq k N(n,p,\frac{(p-1)n}{k})$. Lemma~9.2 of \cite{lovaszsauermann2019multicolor} yields
\[
\left |N\left(n,p,\frac{(p-1)n}{k}\right) \right|\leq (\Gamma_{p,k})^n,
\]
so that $|A|\leq k(\Gamma_{p,k})^n$. To get rid of the factor of $k$, we employ a tensor power trick as in \cite{lovaszsauermann2019multicolor}: for any positive integer $r$, the set
\[
\{(x_{1,j_1},\dots,x_{1,j_r}),\dots,(x_{k,j_1},\dots,x_{k,j_r})\}_{(j_1,\dots,j_r)\in [1,M]^r}
\]
is a $k$-colored sum-free set of size $|A|^r$ in $\FF_p^{nr}$, so that
\[
|A|\leq \lim_{r\to\infty } \sqrt[r]{k(\Gamma_{p,k})^{nr}}=(\Gamma_{p,k})^n,
\]
as claimed.
\end{proof}

\subsection{The finite field Kakeya problem}
Let $\FF=\FF_q$ be a finite field. A set $K\subseteq \FF^n$ is called \emph{Kakeya} if it contains a line in every direction; that is, for any $v\in \FF^n\setminus \{0\}$, we have $\{u_v+a v\mid a\in \FF\}\subseteq K$ for some $u_v\in \FF^n$. In \cite{dvir2009kakeya}, Dvir uses a form of the polynomial method to give the following lower bound on the size of a Kakeya set.

\begin{theorem}[\cite{dvir2009kakeya}]
    \label{thm:kakeya1}
    Let $K\subseteq \FF_q^n$ be a Kakeya set. Then
    \[
        |K|\ge C_n\cdot q^n,
    \]
    where we can take $C_n=\frac{1}{n!}$.
\end{theorem}
The constant $C_n$ in this bound is improved to $(2-\frac{1}{q})^{-n}$ in \cite{dvir2009kakeyamult}, using an extension of the original argument that the authors call the ``method of multiplicities.'' 
\begin{theorem}[\cite{dvir2009kakeyamult}]
    \label{thm:kakeya2}
    If $K\subseteq \FF_q^n$ is a Kakeya set, then $|K|\ge \left(\frac{q}{2-1/q}\right)^n$.
\end{theorem}
More recently, Bukh and Chao \cite{bukhchao2021} improved the lower bound further to $(2-\frac{1}{q})^{-(n-1)}q^n$, making it asymptotically tight. 

As it turns out, the arguments used in \cite{dvir2009kakeya} and \cite{dvir2009kakeyamult} translate quite naturally into the language of shift operators. For purposes of illustration, we will present proofs of both Theorem~\ref{thm:kakeya1} and Theorem~\ref{thm:kakeya2} in the shift operator framework.

For convenience, we define a notion of \emph{directional} Hasse derivatives as follows: For $k\le n$, $I=(v_1,\dots,v_k)$ a sequence of linearly independent vectors in $\FF^n$, and $\alpha=(\alpha_1,\dots,\alpha_k)\in  \ZZ_{\ge 0}^k$, let $H_I^{(\alpha)}:=H_{v_1^{\alpha_k} \cdots v_k^{\alpha_k}}$ be the operator defined by
\[
H_I^{(\alpha)} f(X)=[Z^\alpha] f(X+Z_1 \cdot v_1 +\cdots + Z_k\cdot v_k).
\]
When $k=n$ and $(v_1,\dots,v_n)$ is the standard coordinate basis, this agrees with our usual definition of $H^{(\alpha)}$. We note down a few simple properties of these directional derivative operators, which can be checked via direct computation:
\begin{itemize}
    \item For any $v\in \FF^n$ and $d\ge 0$,
    \[
    H_{v^d}=\sum_{|\alpha|=d}v^\alpha H^{(\alpha)}.
    \]
    \item For linearly independent $v_1,\dots,v_k$,
    \[
    H_{v_1,\dots,v_k}^{(\alpha)}=H_{v_1^{\alpha_1}}\cdots H_{v_k^{\alpha_k}}.
    \]
    \item For $v_1,\dots,v_k\in \FF^n$ linearly independent, $c_1,\dots,c_k\in \FF$, and $d\ge 0$,
    \[
        \label{eqn:dirhasse}
        H_{(c_1v_1+\cdots+c_k v_k)^d}=\sum_{\alpha\in \ZZ_{\ge 0}^k:\: |\alpha|=d} c^\alpha H_{(v_1,\dots,v_k)}^{(\alpha)}. \tag{$2$}
    \]
    
    In particular, this along with the previous property implies that each $H_I^{(\alpha)}$ is a linear combination of the usual Hasse derivatives.
\end{itemize}

With these properties in mind, Theorem~\ref{thm:kakeya1} follows quite readily from considering the invariants $\Delta_K^d$ for a Kakeya set $K$.

\begin{proof}[Proof of Theorem~\ref{thm:kakeya1}]
    Let $K\subseteq \FF^n=\FF_q^n$ be a Kakeya set. For each $v\in \FF^n\setminus \{0\}$, let $L_v=\{u_v+a v\mid a\in \FF\}\subseteq K$ be the set of points in a fixed line in the direction of $v$ contained in $K$. Then Proposition~\ref{prop:1ddl} yields that $\Delta_{L_v}=\bigcup_{0\le d\le q-1}\langle H_{v^d}\rangle$. This means that for each $d\le q-1$, we have
    \[
    \Delta^d_{K}\supseteq \text{span}(\Delta^d_{L_v}: \: v\in \FF^n\setminus\{0\})=\langle H_{v^d}\rangle_{v\in \FF^n\setminus \{0\}}. 
    \]
    In other words, if we consider the span $\Lambda_K$ of the shift operators associated to elements of $K$, the set of lowest degree terms attained in their derivative expansions contains the $d$th directional derivatives in every direction, for all $d\in [0,q-1]$. The key observation now is that the span of these $d$th directional derivatives in fact contains all Hasse derivatives of degree $d$.

    \begin{claim}
        For any $d\in [0,|\FF|-1]$, we have
        \[
        \label{eqn:Hassespan}
        \langle H_{v^d}\rangle_{v\in \FF^n\setminus \{0\}} = \langle H^{(\alpha)}\rangle_{|\alpha|=d}. \tag{$3$}
        \]
    \end{claim}
    \begin{proof}
        We use induction on $n$. The base case $n=1$ is immediate, because $H_{v^d}=v^d H^{(d)}$. Now let $n\ge 2$ and suppose the claim is known for dimension $n-1$, for all $d$. Let $v_1,v_2\in \FF^n$ be linearly independent (and thus, in particular, nonzero). Then for any $c\in \FF$, the left hand side of \eqref{eqn:Hassespan} contains
        \[
            H_{(v_1+cv_2)^d}=\sum_{i=0}^d c^i H_{v_2^i} H_{v_1^{d-i}}.
        \]
        As in the proof of Proposition~\ref{prop:1ddl}, let $v_c=(c^0,\dots,c^d)$. Since $\{H_{v_2^i} H_{v_1^{d-i}}\}_{0\le i\le d}$ is linearly independent, we have $\dim \langle H_{(v_1+cv_2)^d}\rangle_{c\in \FF} = \dim \langle v_c\rangle_{c\in \FF}$. But we showed in the proof of Proposition~\ref{prop:1ddl} that any $d+1$ distinct vectors $v_c$ are linearly independent, so indeed $H_{v_2^i} H_{v_1^{d-i}}\in \langle H_{v^d}\rangle_{v\in \FF^n\setminus \{0\}}$ for any $v_1,v_2$. Fixing $v_1=e_n$ and letting $v_2$ range over all vectors supported on the first $n-1$ coordinates, the inductive hypothesis implies that the span of such operators contains $H^{(\alpha')} H_{e_n}^{(d-i)}$ for any $i$ and any $\alpha'\in \ZZ_{\ge 0}^{n-1}$ with $|\alpha'|=i$, and thus contains $\langle H^{(\alpha)}\rangle_{|\alpha|=d}$ as desired. This completes the induction.
    \end{proof}
    By the claim, we have $\dim(\Delta_K^d)\ge \dim( \langle H^{(\alpha)}\rangle_{|\alpha|=d})=\binom{n+d-1}{n-1}$, and thus
    \[
        |K|\ge \sum_{d=0}^{q-1}\dim(\Delta_K^d) \ge \binom{n+q-1}{n}\ge \frac{1}{n!}q^n,
    \]
    as desired.
\end{proof}

The original proof of Theorem~\ref{thm:kakeya2} builds on Dvir's method in \cite{dvir2009kakeya} by considering polynomials that vanish on $K$ to high multiplicity. We likewise proceed by extending the above argument to the space of shift operators over a \emph{multiset} $(K,m)$.

\begin{proof}[Proof of Theorem~\ref{thm:kakeya2}]
    For ease of comparison with the original proof, we adopt the notation of \cite{dvir2009kakeyamult}: Let $\ell$ be a large multiple of $q$, let $m=2\ell-\ell/q$, and let $d=\ell q -1$. Let $K\subseteq \FF^n$ be a Kakeya set, and consider the multiset $(K,m)$, where the multiplicity function is a constant $m(a)=m$. 
    As before, for each $v\in \FF^n\setminus \{0\}$, let $L_v=\{u_v+a v\mid a\in \FF\}\subseteq K$ be the set of points in a fixed line in the direction of $v$ contained in $K$. Then for any $\beta\in \ZZ_{\ge 0}^n$ with $|\beta|<m$, we have
    \begin{align*}
            \Lambda_{(K,m)}&\supseteq \Lambda_{(L_v,m)}=\langle (T^{h})^{(\gamma)}\rangle_{h\in L_v,\, |\gamma|<m} \\
            &\supseteq \langle H^{(\beta)}(T^h)^{(\alpha)}\rangle_{h\in L_v,\, |\alpha|<m-|\beta|} \supseteq \langle H^{(\beta)}H_{v^{i}} T^h\rangle_{h\in L_v,\, i<m-|\beta|}.
    \end{align*}
    So, by Proposition~\ref{prop:1dml}, $\Delta_{(L_v,m)}\ni H^{(\beta)}H_{v^k}$ for all $k\in [0,q(m-|\beta|)-1]$, $|\beta|<m$. In particular, for any $d^*\le d$, by our choices of constants we have $d^*-|\beta|\le q(m-|\beta|)-1$ whenever $|\beta|\le \ell$, so that
    \[
        \Delta_{(L_v,m)}^{d^*}\ni H^{(\beta)}H_v^{(d^*-|\beta|)} \: \forall \beta\in \ZZ_{\ge 0}^n,\, |\beta|\le \ell. 
    \]
    We now make the following claim.
    \begin{claim}
        \label{claim:multimonindep}
    \[
        \langle H^{(\beta)}H_v^{(d^*-|\beta|)}\rangle_{v\in \FF^n\setminus \{0\},\, |\beta|\le \ell}=\langle H^{(\alpha)}\rangle_{|\alpha|=d^*}.
    \]
    \end{claim}
    \begin{proof}
        It suffices to consider $v=(h_1,\dots,h_n)$ with $h_n=1$, and $\beta=(\beta_1,\dots,\beta_n)$ with $\beta_n=0$. By \eqref{eqn:dirhasse}, we have
    \begin{align*}
                H^{(\beta)}H_v^{(d^*-|\beta|)}&=H^{(\beta)}\left(\sum_{\alpha\in \ZZ_{\ge 0}^n:\: |\alpha|=d^*-|\beta|} h^\alpha H^{(\alpha)}\right) \\
                &= \sum_{\alpha\in \ZZ_{\ge 0}^n:\: |\alpha|=d^*-|\beta|} \binom{\alpha+\beta}{\beta}h^\alpha H^{(\alpha+\beta)} = \sum_{\alpha\in \ZZ_{\ge 0}^n:\: |\alpha|=d^*}(H^{(\beta)}(h^\alpha))H^{(\alpha)}\\
                &=  \sum_{\alpha\in \ZZ_{\ge 0}^{n-1}:\: |\alpha|\le d^*}(H^{(\beta)}(h^\alpha))H^{(\alpha)}H_{e_n}^{(d^*-|\alpha|)}.
    \end{align*}
    It thus suffices to show that the matrix
    \[
        (H^{(\beta)}(h^{\alpha}))_{\substack{(h,\beta)\in (\FF^{n-1},\ell+1) \\ \alpha \in \ZZ_{\ge 0}^{n-1}:\: |\alpha|\le d^*}}\, ,
    \]
    with rows indexed by $(h,\beta)$ and columns indexed by $\alpha$, has full rank. Suppose for the sake of contradiction that this is not the case. Thus, there are some constants $c_\alpha$ such that
    \[
        \sum_{|\alpha|\le d^*} c_\alpha H^{(\beta)} (h^\alpha)=0 \qquad \forall (h,\beta)\in (\FF^{n-1},\ell+1).
    \]
    Let $P(X)=\sum_{|\alpha|\le d^*} c_\alpha X^\alpha$. Then we have $(H^{(\beta)}P)(h)=0$ for all $(h,\beta)\in (\FF^{n-1},\ell+1)$. The generalized Schwartz-Zippel lemma \cite[Lemma~2.7]{dvir2009kakeyamult} then implies that $\deg(P)\ge \frac{|(\FF^{n-1},\ell+1)|}{q^{n-2}}=q(\ell+1)>d^*$, a contradiction. So, the aforementioned matrix indeed has full rank. The number of rows in the matrix is
    \[
        q^{n-1} \binom{\ell+n-1}{n-1}\ge \binom{\ell q+n-1}{n-1}\ge \binom{d^*+n-1}{n-1},
    \]
    where the right hand side is the number of columns. This means that 
    \[
    \dim(\langle H^{(\beta)}H_v^{(d^*-|\beta|)}\rangle_{v\in \FF^n\setminus \{0\}, |\beta|\le \ell})\ge \binom{d^*+n-1}{n-1} = \dim(\langle H^{(\alpha)}\rangle_{ |\alpha|=d^*}).
    \]
    Since the space on the left hand side is a subspace of the space on the right, this shows the desired equality.
    \end{proof}
    By Claim~\ref{claim:multimonindep}, we then have, for each $d^*\le d$,
    \[
        \Delta^{d^*}_{(K,m)}\supseteq \langle H^{(\alpha)}\rangle_{\alpha\in \ZZ_{\ge 0}^n:\: |\alpha|=d^*},
    \]
    so that
    \[
    |K|\binom{m+n-1}{n}=\dim(\Lambda_{(K,m)})\ge \sum_{d^*=0}^d \dim(\Delta_{(K,m)}^{d^*}) \ge \sum_{d^*=0}^d \binom{d^*+n-1}{n-1}=\binom{d+n}{n}.
    \]
    We then conclude, as in \cite{dvir2009kakeyamult}, that
    \[
        |K|\ge \lim_{\ell\to \infty}\frac{\binom{\ell q - 1 + n}{n}}{\binom{2\ell - \ell/q + n-1}{n}}=\left(\frac{q}{2-1/q}\right)^n,
    \]
    as desired.
\end{proof}

\section{Future Directions}

\label{sec:future}

The applications presented in the preceding sections provide a sample, but by no means an exhaustive list, of known results for which our method produces new proofs. In this final section, we shift our attention to discuss a few promising directions in which the shift operator method could be applied to produce new results.

\subsection{The Erd\H{o}s-Ginzburg-Ziv problem}

One direction that holds promise is further study of the Erd\H{o}s-Ginzburg-Ziv problem using this method. As previously mentioned, Zakharov showed in \cite{zakharov2020convexEGZ} that $s(\FF_p^n)\le 4^n p$ for fixed $n$ and sufficiently large $p$, while a recent breakthrough by Sauermann and Zakharov \cite{sauermann2023erd} shows, for each $\varepsilon>0$, an upper bound of the form $D_{\varepsilon, p} \cdot (C_\varepsilon p^{\varepsilon})^n$ in the regime where $p$ is fixed and $n$ grows. It would be very interesting to further improve these upper bounds, whether it is by removing the dependence on $p$ in the base of the exponent for large $n$, obtaining even sharper bounds for fixed $n$, or saying more about bounds in the intermediate regime.

The proofs in Section~\ref{sec:applymd} offer a first piece of evidence that the shift operator method may be relevant for further progress on this problem. However, since lower bound constructions on the order of $\sqrt{p}^n$ exist for the $p$-color sum-free problem over $\FF_p^n$ as discussed in \cite{Sau2023distinct,sauermann2023erd}, any efforts to make progress past this ``multi-colored barrier'' in the Erd\H{o}s-Ginzburg-Ziv problem will need to somehow make use of the fact that multiple copies of the same set are being considered in the sums. The shift operator method suggests a potential avenue for making use of this fact: When sumsets are taken, Lemma~\ref{lem:adddeg} suggests multiplicative behavior in the sets $\Delta_A$ of lowest degree terms in linear combinations of shift operators. When the sumsets consist of multiple copies of $A$ added together, \emph{powers} of these linear combinations become relevant, and it is plausible that these are more fruitful to analyze than general products of such linear combinations.

A more specific line of approach derives from the idea that the invariants $\Delta_A^d$ should capture a great deal of information about the structure of the set $A$ -- as we have seen in the applications to the Nullstellensatz (where they reflect the fact that $A$ contains a grid) and the Kakeya problem (where they directly capture the condition about containing a line in each direction). Concretely, given a ``suitably generic'' set $A$ that is large enough to guarantee high $\deg(A)$ via Proposition~\ref{prop:nddl}(ii), can we in turn guarantee that $\Delta_A^d$ contains every possible lowest degree term of degree $d$, for all small $d$? If so, then as long as we can split a sufficiently large set into, say, $p-1$ suitably generic sets $A_1,\dots,A_{p-1}$, we can guarantee that $H^{(1,\dots,1)}\in \Delta_{A_i}^n$ for each $i$, and thus $H^{(p-1,\dots,p-1)}\in \Delta_{A_1+\cdots +A_{p-1}}^{n(p-1)}$, which would imply $A_1+\cdots +A_{p-1}=\FF_p^n$ by Proposition~\ref{prop:nddl}(i). Perhaps an approach of this type could yield some additional structural constraints on a large enough sequence in $\FF_p^n$ with no $p$-term zero-sum subsequences, if not a general improvement in the known bounds.

\subsection{Sums of Dilates}
Another compelling direction to explore is the application of these methods to the problem of lower bounding the size of a sum of dilates $|\lambda_1 A + \lambda_2 A|$ in $\FF_p$. The corresponding problem over $\ZZ$ is well-studied, culminating in the result of Balog and Shakan in \cite{balogshakandilates} that
\begin{equation}
    \label{eqn:dilatesZ}
    \tag{4}
    |\lambda_1 A + \lambda_2 A|\geq (\lambda_1+\lambda_2)|A|-O_{\lambda_1,\lambda_2}(1),
\end{equation}
for $\lambda_1>\lambda_2\geq 1$, which is sharp up to the additive constant.

The situation in $\FF_p$ is less well understood. Previous work on this problem (see \cite{plagnedilates,pontiverosdilates}) has shown that lower bounds like \eqref{eqn:dilatesZ} can be recovered when $|A|/p$ is sufficiently small that we can transport the problem back into $\ZZ$, a technique called \emph{rectification}. On the other hand, when $|A|/p$ is large, these lower bounds fail. For example, Pontiveros \cite{pontiverosdilates} shows that for any $\lambda$, sets $A$ with density arbitrarily close to $\frac{1}{2}$ can be constructed for sufficiently large $p$ such that $\frac{|A+\lambda A|}{p}$ is bounded away from $1$. It is natural to ask what the transition between these regimes looks like.

For the sake of simplicity, let us focus on the case $(\lambda_1,\lambda_2)=(1,2)$; the observations below hold in general, although the bounds they tease at are less enticing in the general case. Starting as in our proof in Section~\ref{sec:cd}, we can define 
\[
f(z)=\prod_{c\in A+2A} (z-c), \qquad g(x)=\prod_{a\in A}(x-a).
\]
Unlike before, we make the following striking observation: In addition to having $g(z)|T^{2a}f(z)$ for all $a\in A$, we also have $g(z)|T^{a/2}f(2z)$ for all $a\in A$. Thus, if we can find conditions under which most or all of $\{T^{2a}f(z)\}_{a\in A}\cup \{T^{a/2}f(2z)\}_{a\in A}$ is linearly independent, we will have criteria for $\frac{|A+2A|}{|A|}$ to be bounded away from $2$. The task of finding such conditions seems to invite the exploration of scalar multiplication operators like the doubling operator $Z(F)(x):=F(2x)$, and how they interact with the differential-based operators $\partial$ and $T$. In particular, preliminary observations and known constructions seem to hint that linear independence strongly fails here only when $A$ or $A+2A$ has strong multiplicative structure, which hints at potential connections to questions like the one explored in Section~\ref{sec:lacu}.

\section*{Acknowledgements}
The author would like to thank Jacob Fox for inspiring the initial questions that led to this line of research and for many helpful suggestions along the way, and Manik Dhar for inspiring the proof of Theorem~\ref{thm:kakeya1} presented here. The author would also like to thank Ryan Alweiss, Zeev Dvir, Xiaoyu He, Felipe Hernandez, Ray Li, Lisa Sauermann, Yuval Wigderson, Alex Wilson, and potentially others for their helpful input and feedback on various parts of this work.

\bibliographystyle{acm}
\bibliography{main}

\end{document}